\documentclass[a4paper,11pt]{amsart}
\usepackage{mathrsfs}
\usepackage[leqno]{amsmath}
\usepackage{amssymb}
\usepackage[latin1]{inputenc}
\usepackage[T1]{fontenc}

\theoremstyle{plain}

\newtheorem{Prop}[subsection]{Proposition}
\newtheorem{Lem}[subsection]{Lemma}
\newtheorem{Cor}[subsection]{Corollary}

\newtheorem{Rem}[subsection]{Remark}

\newtheorem{de}[subsection]{Definition}

\newcommand{\comments}[1]{}
\newcommand{\ssm}{\smallsetminus}

\newcommand{\osm}{(\Omega, \Sigma, \mu)}

\newcommand{\C}{\mathbb{C}}

\newcommand{\N}{\mathbb{N}}
\newcommand{\R}{\mathbb{R}}

\newcommand{\ga}{\gamma}
\newcommand{\Ga}{\Gamma}

\newcommand{\la}{\lambda}

\newcommand{\Si}{\Sigma}
\newcommand{\si}{\sigma}
\newcommand{\Om}{\Omega}

\newcommand{\lra}{\longrightarrow}

\newcommand{\cL}{{\mathcal L}}
\newcommand{\cM}{{\mathcal M}}

\newcommand{\D}{\mathbb{D}}

\newcommand{\ces}{{\operatorname{ces}\nolimits}}
\newcommand{\cesp}{{\operatorname{ces(\emph{p})}\nolimits}}

\numberwithin{equation}{section}

\begin{document}
\title[Order spectrum of the Cesàro operator in Banach lattice sequence spaces]{Order spectrum of the Cesàro operator in Banach lattice sequence spaces}
\author[J. Bonet and W.J. Ricker]{J. Bonet and W.J. Ricker}
\vspace*{.3cm}

\noindent
\address{J. Bonet, Instituto Universitario de Matemática Pura y Aplicada IUMPA, Universitat Politècnica de València, 46071 Valencia, Spain \newline
Email: jbonet@mat.upv.es}
\address{ W.J. Ricker: Math.-Geogr. Fakultät, Kath. Universität Eichstätt-Ingolstadt, 85072 Eichstätt, Germany\newline
Email: werner.ricker@ku.de }
\maketitle

\begin{abstract}
The discrete Cesàro operator $ C $ acts continuously in various classical Banach sequence spaces within $ \C^{\N}.$
For the coordinatewise order, many such sequence spaces $ X $ are also complex Banach lattices (eg.  $c_0, \ell^p $ for
$ 1 < p \leq \infty , $ and $ \ces (p)$ for $ p \in \{ 0 \}  \cup  ( 1, \infty )).$ In such Banach lattice sequence spaces,
$ C $ is always a positive operator. Hence, its order spectrum is well defined within the Banach algebra of all regular operators
on  $ X .$ The purpose of this note is to show, for every $ X $ belonging to the above list of Banach lattice sequence spaces, that the order
spectrum $ \si_{\rm o} (C)$ of $ C $  \textit{coincides} with its usual spectrum $ \si ( C)$ when $ C $ is considered as a
continuous linear operator on the Banach space $ X .$
\end{abstract}

\section{Introduction} \label{S1}
Let $ E $ be a complex Banach lattice and $ \cL (E)$ denote the unital Banach algebra of all continuous linear operators from $ E $
into itself, equipped with the operator norm $ \| \cdot \|_{{\rm op}}.$ The unit is the identity operator $ I : E \lra E .$
Associated  with each $ T \in \cL (E)$ is its \textit{spectrum}\/
$$  \textstyle
\si (T) := \{ \la \in \C : ( \la I - T )\  \mbox{ is not invertible in } \
 \cL (E) \}
 $$
 and its resolvent set $ \rho (T) := \C\!\ssm\!\si(T).$ An operator $ T \in \cL (E)$ is called \textit{regular}\/
 if it is a  finite linear combination of \textit{positive operators}. The complex vector space of all regular operators is denoted
 by $ \cL ^r (E);$ it is also a unital Banach algebra for the norm
 \begin{equation} \label{1.1}
  \| T \|_r := \inf \{ \| S \|_{\rm op} : S \in \cL (E), S \geq 0 , | T (z)| \leq S (| z | ) \; \forall \: z \in E \} , \quad T \in \cL ^r (E).
 \end{equation}
 Again $ I  : E \lra E $ is the unit.
 Moreover, $  \| T \|_{\rm op} \leq \| T \|_r $ for $ T \in \cL ^r (E) ,$ with equality whenever $ T \geq 0 $ (i.e., if $ T $ is a positive
 operator). The spectrum of $ T  \in \cL ^r (E),$ considered as an element of the Banach algebra $ \cL ^r (E),$ is denoted by
 $ \si_{\rm o} (T) $ and is called its \textit{order spectrum}. Then $ \rho_{\rm o} (T) := \C\!\ssm\!\si_{\rm o} (T)$ is the \textit{order resolvent}\/ of
 $ T .$ Clearly
 \begin{equation} \label{1.2}
 \si (T) \subseteq \si_{\rm o} (T), \quad T \in \cL^r (E).
 \end{equation}
 From the usual  formula for the spectral radius,  \cite[Ch.I, \S 2, Proposition 8]{BD}, it follows that the spectral
 radii for $ T \in \cL ^r (E) $ satisfy $ r (T) = r _{\rm o} (T)$ whenever $ T \geq 0.$ Standard references for the above concepts and facts are
 \cite{A}, \cite{Sch1}, \cite{Sch2}, for example.

 It is clear from \eqref{1.2} that $ r (T) \leq r_{\rm o} (T)$ for $ T \in \cL^r (E).$ So, if $ r (T) < r_{\rm o}(T),$ then \eqref{1.2} cannot be an equality.
 This is the strategy applied in \cite[pp.79-80]{Sch2} to exhibit a regular operator for which $ \si (T) \subsetneqq \si_{\rm o} (T).$ For an
 example of a \textit{positive operator} $ T $ satisfying $ \si (T) \subsetneqq \si_{\rm o} (T),$ see \cite[pp.283-284]{A}. In the contrary direction, a
 rich supply of classical operators $ T $ for which the equality
 \begin{equation} \label{1.3}
  \si (T) = \si_{\rm o} (T)
  \end{equation}
  is satisfied arise in harmonic analysis,  \cite[Theorem 3.4]{A}.\\

   The aim of this note is to contribute two further classes of
  operators $ T $ which satisfy \eqref{1.3}. In Section \ref{S2} it is shown that in any \textit{Banach function space} $ E ,$ all
  multiplication operators $ T $ by $ L^\infty $-functions are regular operators and satisfy  \eqref{1.3}. This is a consequence of the fact that the algebra
  of such multiplication operators is maximal commutative. Let $  \N := \{ 1, 2, \ldots \} .$
   The remaining three sections deal with the classical \textit{Cesàro operator} $ C : \C^{\N} \lra \C^{\N} $
  defined by
  \begin{equation} \label{1.4}  \textstyle
  C (x) := \left( \frac{1}{n} \sum^n_{ k = 1} x_k \right) ^\infty _{ n = 1} \quad x = ( x_n) ^\infty _{ n = 1} \in \C^{\N} ,
  \end{equation}
  which is clearly a \textit{positive operator}\/ for the coordinatewise order in the positive cone of $ \C^{\N} = \R^{\N} \oplus i \R^{\N}.$ Section \ref{S3}
  establishes some general results for determining the regularity  of linear operators in \textit{Banach lattice sequence spaces}. These results are designed to apply
  to the particular operators $ (C - \la I )^{-1},  $ where  $ C $ is  given in \eqref{1.4}. In Section \ref{S4}
  we will consider the restriction of $ C $ to the Banach lattice sequence spaces $ c_0 $ and $ \ell^p , 1 < p \leq \infty ,$ and show that \eqref{1.3} is satisfied in all
  cases (with $ C$ in place of $ T $). Section  \ref{S5} is devoted to proving the same fact, but now when $ C $ acts in the discrete Cesàro spaces $  \cesp,
  1 < p < \infty, $ and in $ \ces(0).$

  \section{Multiplication operators} \label{S2}

  Let $ \osm $ be a \textit{localizable measure space} (in the sense of \cite[64A]{F}),  that is, the associated measure algebra
  is a complete Boolean algebra and, for every measurable set $ A \in \Si $ with $ \mu (A) > 0 $  there exists $ B \in \Si $
  such that $ B \subseteq A $ and $ 0 < \mu (B) < \infty $ (i.e., $  \mu $ has the finite subset property). All $ \si$-finite measures are localizable,
  \cite[64H Proposition]{F}.
  Every Banach function space $ E $ (of $ \C $-valued functions) over $ \osm$ is a complex Banach lattice for the pointwise
  $\mu$-a.e. order. Given any  $ \varphi \in L^\infty (\mu),$ the multiplication operator $ M_\varphi : E \lra E $ defined by
   $ f \longmapsto
  \varphi f ,$ for $ f \in E ,$ belongs to $ \cL (E)$ and satisfies $ \| M_\varphi \| _{\rm op} = \| \varphi \|_{\infty}  .$
Define a unital, commutative subalgebra of $ \cL ( E) $ by
$$  \textstyle
\cM _E (L^\infty (\mu)) := \{ M_\varphi : \varphi \in L^\infty ( \mu) \} ;
$$
the unit is the identity operator $ I = M_{\mathbf 1} $ where $ \mathbf{1} $ is the constant function 1 on $ \Om.$ Recall that the \textit{commutant}\/
of $ \cM_E (L^\infty (\mu))$ is defined by
$$
\cM _E (L^\infty (\mu))^c := \{ A \in \cL (E)
 : A M_\varphi = M_\varphi A \;\: \forall  \, \varphi \in L^\infty  (\mu) \} \subseteq \cL (E) .
 $$
 It is known that $ \cM_ E (L^\infty ( \mu)) $ is a \textit{maximal commutatitive}, unital subalgebra of $ \cL (E),$ that is,
 $ \cM_E ( L^\infty ( \mu)) = \cM _E ( L^\infty ( \mu))^c ,$ \cite[Proposition 2.2]{dPR}. Moreover, also the \textit{bicommutant}
 $ \cM_E ( L^\infty (\mu))^{cc} = \cM_E (L^\infty ( \mu)). $
 \begin{Prop}  \label{P2.1} \  Let $ \osm$ be a localizable measure space and $ E $ be  a Banach function space over $ \osm.$
 \begin{itemize}
 \item[{\rm (i)}] $\cM_E ( L^\infty ( \mu)) \subseteq \cL^r (E) .$
 \item[{\rm (ii)}]  $ \cM_E ( L^\infty ( \mu))$ is inverse closed in $ \cL (E).$ That is, if $ T \in \cM _E (L^\infty ( \mu)) $
 is invertible in $ \cL (E) $ (i.e., there exists $ S \in \cL (E) $ satisfying $ ST = I = TS),$ then necessarily
 $ S \in \cM _E ( L^\infty ( \mu)).$
 \item[{\rm (iii)}] For every $ T \in \cM_E ( L^\infty (\mu)) $ we have $ \si_{\rm o} (T) = \si (T).$
 \end{itemize}
 \end{Prop}
\begin{proof} \ (i) \ Let $ \varphi \in L^\infty (\mu).$ Then $ \varphi =   [({\rm Re} \,  \varphi)^+  - ({\rm Re} \, \varphi)^- ] +  i   [ ({\rm Im} \, \varphi)^+
- (\mbox{Im} \, \varphi)^- ] $ with all four functions $ (\rm{Re} \, \varphi)^+, (\rm{Re} \,  \varphi)^-, (\rm{Im} \,  \varphi)^+, (\rm{Im} \,  \varphi)^- $
belonging to the positive cone $ L^\infty ( \mu)^+ $ of $ L^\infty (\mu).$ Since
$
M_\varphi = [ M_{(\rm{Re} \,  \varphi )^+} - M_{(\rm{Re} \,  \varphi)^-} ] + i [ M_{(\rm{Im} \, \varphi)^+}  - M_{(\rm{Im} \,  \varphi)^-} ]
$  is a linear combination of positive operators,
it is clear that $ M_\varphi \in \cL^r (E).$

(ii) \ Since $ \cM _ E (L^\infty (\mu))$ is maximal commutative in $ \cL (E), $ it follows  that $ \cM _E (L^\infty (\mu))$ is inverse closed in $ \cL ( E), $
 \cite[Ch.II, \S 15, Theorem 4]{BD}.

 (iii) \ In view of \eqref{1.1} it suffices to show that $ \rho ( T) \subseteq \rho_{\rm o} (T).$ Suppose that $ T = M_\varphi $
 with $ \varphi \in L^\infty ( \mu).$ Fix $ \lambda \in \rho (T).$ Then $ \lambda I - T = M_{ ( \lambda {\mathbf 1}  - \varphi)} $ belongs to $ \cM _E ( L^\infty (\mu))$
 because $ ( \lambda {\mathbf 1}  - \varphi ) \in L^\infty (\mu). $ Since $ M_{ ( \lambda {\mathbf 1} - \varphi )} $ is invertible in $ \cL (E),$ it follows from part (ii) that actually $ ( \lambda I - T)^{-1}
 \in \cM_E ( L^\infty ( \mu))$ and hence, by part (i), that also $ ( \lambda I - T )^{-1} \in \cL ^r (E) .$
 \end{proof}

 \begin{Rem} \label{R2.2} {\rm We point out that  $ \| T \|_{\rm op} = \| T \|_r $ for each $ T \in \cM_E (L^\infty (\mu)).$
 Indeed, let $ \varphi \in L^\infty ( \mu) $ satisfy $ T = M_\varphi ,$ in which case $ \| M_\varphi \|_{\rm op} = \|\varphi \|_\infty.$
 Define $ S := \| \varphi \| _\infty I $ and note that $ S \geq 0 $ with $ \| S \|_{\rm op} = \| \varphi \| _\infty .$ Moreover,
 $$ \textstyle
  | M _\varphi (f) | = | \varphi f | \leq \| \varphi \| _\infty | f | =  S (| f | ), \quad f \in E ,
  $$
 and so $ \| T \|_r \leq \| S \|_{\rm op} = \| \varphi \| _\infty = \| T  \|_{\rm op} ;$ see \eqref{1.1}. The reverse inequality
 $ \| T \|_{\rm op} \leq \| T \|_r $ always holds.  }
 \end{Rem}

 \section{The Cesàro operator in Banach sequence spaces} \label{S3}

 We begin with some preliminaries. Equipped with the topology of pointwise convergence $ \C^{\N} $ is a locally convex Fréchet space. Let $ A = ( a_{nm})^\infty _{n, m = 1} $
 be any lower triangular (infinite) matrix, i.e., $ a_{nm} = 0 $ whenever $ m > n .$ Then $ A $ induces the continuous linear operator $ T _A : \C^{\N} \lra \C ^{\N} $
 defined by
 \begin{equation}  \label{3.1} \textstyle
 T_A (x) := \left( \sum ^\infty_{ m = 1} a_{nm} x_m \right) ^\infty _{ n = 1}, \quad x \in \C^{\N}.
 \end{equation}
 For $ x \in \C^{\N} $ define $ | x | := ( | x _n | )^\infty _{ n = 1} .$ Then also $ | x | \in \C ^{\N} .$  A vector subspace $ X \subseteq \C^{\N} $ is
 called \textit{solid}\/  (or an \textit{ideal}) if $ y \in X $ whenever $ x \in X $ and $ y \in \C^{\N} $ satisfy $ | y | \leq | x | .$ It is always
 assumed that $ X $ contains the vector space consisting of all elements of $ \C^{\N} $ which have only finitely many non-zero coordinates. In
 addition, it is assumed that $ X $ has a norm $ \| \cdot \| _X $ with respect to which it is  a complex  \textit{Banach lattice}\/ for the
 \textit{coordinatewise order}\/ and such that the natural inclusion $ X  \subseteq \C^{\N} $ is continuous. Under the previous requirements $ X $
 is called a \textit{Banach lattice  sequence space}.

 \begin{Lem} \label{L3.1} Let $ A = ( a_{nm})^\infty_{ n, m = 1} $ be a lower triangular matrix with all entries non-negative real numbers  and
 $ X \subseteq \C^{\N} $ be a Banach lattice  sequence space such that $ T_A (X) \subseteq X.$ Let $ B = ( b_{nm} )^\infty_{ n, m = 1} $ be any
 matrix such that
 \begin{equation} \label{3.2}
  | b_{nm}| \leq a_{nm} , \quad  n, m \in \N .
  \end{equation}
  Then the restricted operator $ T_A   : X \lra X $ belongs to $ \cL ( X).$ Moreover, $ T_B : \C^{\N} \lra \C^{\N} $ satisfies  $ T_B (X) \subseteq X $ and
  the restricted operator $ T_B  : X  \lra X $ also belongs to  $ \cL (X). $  In addition,  $ \| T_B \|_{\rm op} \leq \| T _A \| _{\rm op} .$
  \end{Lem}
\begin{proof} \
Condition \eqref{3.2}  implies that $ B$ is also a lower triangular matrix. Moreover, the continuity of  both $ T_A : \C^{\N}  \lra \C^{\N} $ and of the inclusion map  $ X \subseteq \C^{\N} $
imply, via the Closed Graph Theorem in the Banach space  $ X ,$ that the restricted operator $ T_A \in \cL (X).$

Given  $ x \in X $ we have for each $ n \in \N ,$  via \eqref{3.2}, that
$$ \textstyle
 (T_B (x))_n =  | \sum^\infty_{ m = 1} b_{nm} x_m | \leq \sum^\infty_{ m = 1} | b_{nm}| \cdot | x_m |  \leq \sum ^\infty _{ m = 1} a_{nm} | x _m|  = (T_A ( | x |))_n .
 $$
 Since $ X $ is solid and $ T_ A (| x | ) \in X,$ these inequalities and \eqref{3.1} imply that $ T_B (x) \in X.$ Moreover, as $ \| \cdot \| _X $ is a lattice norm it
 follows that
 $$ \begin{array}{lcl}
 \| T_B (x) \|_X  & = & \| ( \sum ^\infty _{ m = 1} b_{nm} x _m ) ^\infty_{n = 1} \| _X

  \;  \leq \;   \| ( \sum^\infty_{ m = 1}  a_{nm}  | x_m | ) ^\infty_{ n = 1 } \| _X \\[1.5ex]

  & = &  \| T_A ( | x | ) \|_X \;  \leq \;  \| T_A \|_{\rm op} \| x \| _X ,
 \end{array}
 $$
 for each $ x \in X, $  where the stated series are actually finite sums.
 Hence, $ \| T_B\|_{\rm op} \leq \| T_A \| _{\rm op} $ and the proof is complete.
 \end{proof}

 Since the operator $ T_A $ as given in Lemma \ref{L3.1} satisfies $ T_A \geq 0 ,$ it is clearly regular.

 \begin{Cor} \label{C3.2}
 Let $ A  = ( a_{nm} )^\infty _{n, m = 1} $ be a lower triangular matrix with non-negative real  entries and $ X \subseteq \C^{\N} $ be a
 Banach lattice sequence space such that $ T_A ( X) \subseteq X.$ Let $ B = ( b_{nm} )^\infty_{n,m = 1} $ be any matrix satisfying
 \eqref{3.2}. Then the operator $ T_B \in \cL (X)$ is necessarily regular, that is, $ T  _B \in \cL^r (X).$
 \end{Cor}
 \begin{proof}  \  Define the non-negative real numbers $ s_{nm} := (\mbox{Re} \, b_{nm})^+, u_{nm} := (\mbox{Re} \,  b_{nm})^-,
 v _{nm} := (\mbox{Im}\,  b_{nm})^+ $ and $ w_{nm} := (\mbox{Im} \,  b_{nm})^- $ for each $ n , m \in \N. $ Then
 $ b_{nm} = (s_{nm} - u_{nm}) + i (v_{nm} - w_{nm} ) $ and  $ \{ s_{nm}, u_{nm}, v_{nm}, w_{nm}  \} \subseteq [0, a_{nm}] $
 for $ n , m \in \N.$ Setting $ S := ( s_{nm} )^\infty_{n, m = 1} , U := ( u_{nm} ) ^\infty_{ n,m = 1} , V := (v_{nm} )^\infty_{n,m = 1} $ and
 $ W := (w_{nm} )^\infty _{n, m = 1} $ it is clear from the definition  \eqref{3.1} that each operator
 $ T_S \geq 0, T_U \geq 0, T_V \geq 0 $ and  $ T_W \geq 0 $ (in $X$) belongs to $ \cL (X) ;$
 see Lemma  \ref{L3.1}. Since $ T_B = (T_S - T_U) + i (T_V - T_W ),$ it follows that $ T_B \in \cL^r (X).$
 \end{proof}
 Together with appropriate estimates, Corollary \ref{C3.2}  will be the main ingredient required to establish \eqref{1.3}
 for $ C $ (in place of $ T $) when it acts  in various classical Banach lattice sequence spaces $ X .$

 Let $ \Si_0 := \{ 0 \} \cup \{ \frac{1}{n}  : n \in \N \}.$ We recall the formula for the inverses
 $ (C - \la I )^{-1} : \C^{\N} \lra \C^{\N} $ whenever $ \lambda \in \C\!\ssm\!\Si_0,$ \cite[p.266]{R}.
 Namely, for $ n \in \N $ the $n$-th row of the lower triangular matrix determining
 $ (C - \la I )^{-1} $ has the entries
 \begin{equation} \label{3.3}  \textstyle
 \frac{-1}{n \la^2 \prod^n _{k = m} ( 1 - \frac{1}{k \lambda})} , \quad
 1 \leq m < n , \quad   \mbox{ and }  \quad
  \frac{n}{1 - n \lambda} = \frac{1}{ ( \frac{1}{n} - \la)} , \quad m = n ,
  \end{equation}
  with all other entries in row $ n$ being 0. We write
  \begin{equation} \label{3.4} \textstyle
   (C - \la I )^{-1} = T_{D_\la} - \frac{1}{\la^2} T_{E_\la} ,
   \end{equation}
   where the diagonal matrix $ D_\la = (d_{nm} ( \la))^\infty_{ n,m = 1} $ is given by
   \begin{equation} \label{3.5}  \textstyle
    d_{nn}  ( \la ) := \frac{1}{  ( \frac{1}{n} - \la ) } \quad \mbox{ and } \quad d_{nm} ( \la ) := 0 \quad \mbox{ if }  \:
    n \neq m .
\end{equation}
Setting $ \gamma [\la] := \mbox{dist} ( \la, \Si_0 ) > 0$ it is routine to check that
\begin{equation} \label{3.6} \textstyle
 | d_{nn}  ( \la ) | \leq \frac{1}{\ga [\la]} , \quad n \in \N, \quad \la \in \C\!\ssm\!\Si_0 .
 \end{equation}
 Moreover, $ E_\la = ( e_{nm}  ( \la ))^\infty_{ n, m = 1} $ is the lower triangular matrix given by
 $ e_{1m} ( \la) = 0, $ for $ m \in \N , $ and for all $ n \geq 2  $ by
\begin{equation} \label{3.7}  \textstyle
 e_{nm} ( \la) := \left\{ \begin{array}{ccl}
  \frac{1}{n \,   \Pi ^n _{k = m} (1 - \frac{1}{k \la})}  \quad & \mbox{if}  \quad & 1 \leq m < n  \\[1.5ex]

  0 & \mbox{if} \quad & m \geq n .
  \end{array}  \right. \end{equation}

  \begin{Lem} \label{L3.3}
   Let $ X \subseteq \C^{\N} $ be any Banach lattice  sequence space. For each $  \la \in \C\!\ssm\!\Si_0$ the diagonal operator $ T_{D_\la} , $ with
   $ D_\la  = (d_{nm} (  \la ))^\infty_{ n,m = 1} $ given by \eqref{3.5}, is regular in $ X ,$ that is,
   $ T_{D_\la} \in \cL ^r ( X).$
   \end{Lem}
   \begin{proof} \  Fix $ \la \not\in \Si_0$ and let $ A := \frac{1}{\ga [\la]} I,$ where $ I $ is the identity matrix in $ \C^{\N} ,$
   in which case $ T_A (X) \subseteq X$ is clear. It follows from \eqref{3.6} that the matrix $ B := D_\la $ satisfies \eqref{3.2}.
   Hence, the regularity of $ T_{D_\la} $ in $ X $ follows from Corollary  \ref{C3.2}.
   \end{proof}
   \begin{Rem} \label{R3.4} {\rm \ (i) \ Since any Banach lattice  sequence space $ X  \subseteq \C^{\N} $ is a  Banach function space over the $ \sigma$-finite
   measure space $ ( \N, 2 ^{\N} , \mu),$ relative to counting measure $ \mu,$ and the function $ n \longmapsto d_{nn}  ( \la)$
   on $ \N $ belongs to $ L^\infty ( \mu) $ by \eqref{3.6}, the regularity of $ T_{D_\la} \in \cL (X) $ also follows from
   Proposition \ref{P2.1}(i).

   (ii) \ For appropriate $ X $ and $ \la \not\in \Si_0,$ it is clear from \eqref{3.4} and Lemma \ref{L3.3}  that the regularity of
   $ (C - \la I )^ {-1}  \in \cL (X) $ is completely determined by the matrix  $ E_ \la .$

   The following inequalities will be needed in the sequel. For $ \alpha < 1 $ we refer to \cite[Lemma 7]{R} and for general $ \alpha \in \R $
   to \cite[Lemma 3.2(i)]{ABR1}. }
   \end{Rem}
  \begin{Lem} \label{L3.5} Let $ \la \in \C\!\ssm\!\Si_0 $ and set $ \alpha := \mbox{Re} ( \frac{1}{\la}).$ Then there exist positive
  constants $ P ( \alpha) $  and $ Q ( \alpha) $ such that
  \begin{equation} \label{3.8} \textstyle
   \frac{ P ( \alpha)}{n ^\alpha} \leq \prod^n_{ k = 1} | 1 - \frac{1}{k \la} | \leq \frac{Q ( \alpha)}{ n^\alpha} , \quad n \in \N  .
   \end{equation}
 \end{Lem}
 \section{The classical spaces $ \ell^p, 1 < p \leq \infty,$ and $ c_0$}  \label{S4}
 For each $ 1 < p \leq \infty $ let $ C_p \in \cL ( \ell^p )$ denote the Cesàro operator as given by \eqref{1.4}
 when it is restricted to $ \ell^p .$ As a consequence of  Hardy's inequality, \cite[Theorem 326]{HLP}, it is known that
 $ \| C_p \|_{\rm op} = p', $ where $ \frac{1}{p} + \frac{1}{p'} = 1  $ (with $ p': = 1 $ when $ p = \infty ).$
 Concerning the spectrum of $ C_p $ we have
 \begin{equation}  \label{4.1} \textstyle
  \si ( C_p ) = \{ \la \in \C : | \la - \frac{p'}{2} | \leq \frac{p'}{2} \}, \quad 1 < p \leq \infty .
  \end{equation}
  Various proofs of \eqref{4.1}  are known for $ 1 < p < \infty , $ \cite{CR1},  \cite{L1}, \cite{L2}, \cite{Rh1}, \cite{Rh2};
  see the discussion on p.268 of \cite{CR1}. For the case $ p = \infty $ we refer to \cite[Theorem 4]{L1}, for example.

  \begin{Rem} \label{R4.1}  {\rm
  For each $ \la \neq 0 $ set $ \alpha := \mbox{Re}   ( \frac{1}{\la} ).$ Then, for any $ b > 0 $ we have
  $$    \textstyle
  \alpha < \frac{1}{b} \: \mbox{ and only if } \:
   | \la - \frac{b}{2}| > \frac{b}{2}.
  $$
  The corresponding results for $ \alpha > \frac{1}{b} $ and $ \alpha = \frac{1}{b} $ also hold. }
  \end{Rem}

  \begin{Prop} \label{P4.2} \ For each $ 1 < p < \infty $ the order spectrum of the positive operator $ C_p \in \cL ( \ell^p ) $ satisfies
  \begin{equation}  \label{4.2}
   \si_ {\rm o} ( C_p ) = \si ( C_p ) .
 \end{equation}
 \end{Prop}
 \begin{proof}  \  Via  \eqref{1.2} it suffices to verify that $ \rho ( C_p ) \subseteq \rho_{\rm o} (C_p). $\\

 With the notation of \eqref{3.4} and \eqref{3.7} it is shown on p.269 of \cite{CR1}, as a consequence
 of \eqref{3.8} in Lemma \ref{L3.5} above, that for every $ \la \neq 0 $ satisfying   $ \alpha := \mbox{Re}(\frac{1}{\la}) < 1 $  there exists a constant
 $ \beta ( \la ) > 0 $ such that
 \begin{equation} \label{4.3} \textstyle
 | e_{nm} (  \la ) | \leq \frac{\beta ( \la)}{ n^{1 - \alpha} m^\alpha} , \quad 1 \leq m \leq n , \quad n \in \N .
 \end{equation}
 Set $ B := E_\la $ and let $ A $ be the lower triangular matrix whose entries $ a _{nm} ( \la ) \geq 0 $ are given
 by the right-side of \eqref{4.3} for each $ n \in \N $ and
 $ 1 \leq m \leq n $ (and 0 otherwise). According to \eqref{4.3} the  matrices $ A $ and $ B $ satisfy \eqref{3.2}.
 Let  $ X := \ell^p $ for $ p \in (1, \infty )$ fixed. Then Corollary \ref{C3.2} implies that $ E_\la $ will be
 regular (i.e.,  $ T_{ E_\la} \in \cL^r ( \ell^p))$ whenever $ T_A ( \ell^p ) \subseteq \ell^p .$ Note that $ T_A \in \cL ( \C^{\N} ) $
 is given by
\begin{equation} \label{4.4} \textstyle
x \longmapsto \beta ( \la) \left( \frac{1}{n ^{1 - \alpha}} \sum^n_{ m = 1} \frac{x_m}{m^\alpha} \right) ^\infty_{ n = 1}  :=
\beta ( \la ) G_\la ( x ), \quad x \in \C^{\N}.
\end{equation}
So, if $ \mbox{Re}( \frac{1}{\la}) < 1 ,$ then \eqref{4.4} implies that $ T_A\in \cL ( \ell ^p )$ whenever
$ G_\la : \ell ^p  \lra \ell ^p $ is continuous.

Let now $ \la \in \rho ( C_p), $ that is, $  | \la - \frac{p'}{2}| > \frac{p'}{2} .$ Then
$ \alpha := \mbox{Re}( \frac{1}{\la}) < \frac{1}{p'},$ because of  Remark \ref{4.1}, and hence,  $ (1 - \alpha ) p > 1 .$
Then the Proposition  on p.269 of \cite{CR1} yields that indeed $ G_\la \in \cL (  \ell^p ). $
As noted above, this
implies that $ T_{E_\la}  \in \cL ^r ( \ell^ p ).$
Combined with \eqref{3.4} and Lemma \ref{L3.3} it follows that $ ( C_p - \la I )^{-1} \in \cL^r ( \ell^p ),$
  that is, $ \la \in \rho_{\rm o } ( C_p).$ This completes the proof of \eqref{4.2}.
\end{proof}
Recall that $ \| C_\infty \| _{\rm op} = 1 $ and,  from \eqref{4.1} for $ p = \infty ,$ that
\begin{equation} \label{4.5} \textstyle
\si ( C_\infty ) = \{ \la \in \C : | \la - \frac{1}{2}| \leq \frac{1}{2} \}
\end{equation}
\begin{Prop} \label{P4.3} The order spectrum of the positive operator  $ C_\infty \in \cL ( \ell^\infty ) $ satisfies
$$ \si_{\rm o} ( C_\infty ) = \si (C_\infty ).
$$
\end{Prop}
\begin{proof} \  Again by \eqref{1.2} it suffices to prove that  $ \rho ( C_\infty ) \subseteq \rho_{\rm o} ( C_\infty ). $ \\

Fix $ \lambda \in \rho ( C_\infty ).$ According to \eqref{4.5}, for $ b = 1 $  the condition  in Remark \ref{R4.1} is  satisfied with
$ \alpha : = \mbox{Re} ( \frac{1}{\la}).$ Hence, the inequalities \eqref{4.3} are valid and so $ A := ( a_{nm} ( \la))^\infty _{n,m = 1} \geq 0 $
and $ B := E_\la $ can again be defined exactly  as in the proof of Proposition \ref{4.2}. Then \eqref{3.2} is satisfied with $ X := \ell^\infty.$ Arguing as in the
proof of Proposition \ref{P4.2}  (via Corollary \ref{C3.2}) it remains to verify that $ T_A : \ell ^\infty \lra \ell ^\infty $ is continuous,
where $ T _A $ is given by \eqref{4.4}. To this effect, since  $ (1 - \alpha ) > 0$ by Remark \ref{R4.1},  it follows  that
 \begin{equation} \label{4.6} \textstyle
\sup _{n \in \N} \sum^\infty _{ m = 1} | a_{nm} ( \lambda ) | = \beta ( \lambda ) \sup _{n \in \N}
\frac{1}{n^{1 - \alpha}}  \sum ^\infty_{ m = 1} \frac{1}{m^\alpha}  < \infty ;
 \end{equation}
 this has been verified on p.778 of \cite{ABR} (put $ w (n) = 1$  there for all $ n \in \N )$
 by considering each of the cases $ \alpha < 0 ,  \, \alpha = 0 $ and $ 0 < \alpha < 1 $ separately. But, condition
 \eqref{4.6} is known to imply that $ T_A \in \cL ( \ell^\infty ),$  \cite[Ex.2, p.220]{T}. The  proof that
 $ \lambda \in \rho_{\rm o} (C_\infty ) $ is thereby complete.
\end{proof}
To conclude this section we consider the Cesàro operator $ C ,$ as given by \eqref{1.4}, when it is restricted to $ c_0;$
denote this operator by $ C_ 0.$ It is shown in \cite[Theorem 3]{L1}, \cite{R}, that $ \| C_0 \| _{\rm op} = 1 $ and
\begin{equation} \label{4.7} \textstyle
\si ( C_0) = \{ \lambda \in \C : | \lambda  - \frac{1}{2} | \leq \frac{1}{2}  \} .
\end{equation}
\begin{Prop} \label{P4.4} The order spectrum of the positive operator  $ C_0 \in \cL (c_0 )  $ satisfies
$$
\si _{\rm o } ( C_0 ) = \si ( C_0 ) .
$$
\end{Prop}
\begin{proof} \  Since \eqref{4.7} shows that $ \si ( C_0 ) = \si ( C_\infty),$ the entire proof of Proposition \ref{P4.3}
can be easily adapted (now for $ X := c_0 $ and fixed  $ \lambda \in  \rho ( C_0 )),$  using the same notation, \textit{ up to the stage}\/  where
\eqref{4.6} is shown to be valid. In \textit{addition}\/ to the validity of \eqref{4.6}  it is also true
 that
\begin{equation} \label{4.9} \textstyle
\lim_{n \to \infty} a_{nm} ( \lambda ) = \frac{\beta ( \lambda)}{m^\alpha}  \lim_{n \to \infty} \frac{1}{n^{1 - \alpha}} = 0, \quad m \in \N ,
\end{equation}
because $ \alpha := \mbox{Re} ( \frac{1}{\lambda}) $ satisfies $ ( 1 - \alpha ) > 0 .$ The two conditions \eqref{4.6} and \eqref{4.9}
together are known to imply that $ T_A \in \cL ( c_0 ),$ \cite[Theorem 4.51-C]{T}. Again via Corollary \ref{C3.2} and Lemma \ref{L3.3}
we can conclude that $ T_{E_\lambda} \in \cL ^r ( c_0) $ and hence, also $ (C_0  - \lambda I )^{-1} $ is regular on $ c_0 .$
\end{proof}

\section{The discrete Cesàro spaces $ \cesp, 1 < p < \infty ,$  and  $ \ces (0)$ } \label{S5}

For $ 1 < p < \infty $ the discrete Cesàro spaces are defined  by
$$ \textstyle
\cesp  := \{ x \in \C^{\N} : \| x \|_{\cesp} := \left( \sum^\infty _{n = 1} ( \frac{1}{n} \sum^n _{k = 1} | x_k |) ^p \right) ^{1/p} < \infty \}  .
$$
In view of \eqref{1.4} we see that $ \| x \|_ {\cesp}  = \| C ( | x | )\|_{\ell^p} $ for $ x \in \cesp . $ It is known that each space
$ \cesp , 1 < p < \infty ,$  is a reflexive Banach lattice sequence space for the norm $ \| \cdot \| _{\cesp} $ and the coordinatewise order. The spaces
$ \cesp $ have been thoroughly treated in \cite{B}. According to Theorem \ref{5.1} of \cite{CR3} the restriction of $ C $
(see \eqref{1.4}) to $ \cesp ,$ denoted here by $ C_{(p)},$ is continuous with $ \|  C_{(p)} \| _{\rm op} =  p'$ and
\begin{equation} \label{5.1} \textstyle
 \si ( C_{(p)} ) = \left\{ \lambda \in \C : | \lambda - \frac{p'}{2} | \leq \frac{p'}{2} \right\}, \quad  1 < p <  \infty .
 \end{equation}
 \begin{Prop} \label{P5.1}
 For each $ 1 < p < \infty $ the order spectrum of the positive operator  $ C_{(p)}  \in \cL ( \cesp) $ satisfies
 \begin{equation} \label{5.2} \textstyle
  \si_ {\rm o} ( C_{(p)} ) =  \si ( C_{(p)}) .
  \end{equation}
\end{Prop}
\begin{proof} \  In view of \eqref{1.2} it suffices to verify that $ \rho ( C_{(p)}) \subseteq \rho_{\rm o} ( C_{(p)}).$\\

  We decompose the set
$ \rho  ( C_{(p)}) $ into two disjoint parts, namely the set
\begin{equation}  \label{5.3} \textstyle
 \rho_1 := \{ \la \in \C\!\ssm\!\{0 \} : \mbox{ Re} \left( \frac{1}{\la} \right) \leq 0 \} = \{ u \in \C\!\ssm\!\{ 0 \} : \mbox{  Re}(u) \leq 0 \}
 \end{equation}
 and  its complement $ \rho_2 := \rho ( C_{(p)} )\!\ssm\!\rho_1 .$\\

 First fix $ \la \in \rho_1 .$ Then $ \la \not\in \Si_0 $ and so we may consider $ E_\la = ( e_{nm} ( \la ))^\infty_{ n, m = 1}  $
 and $ D_\la = ( d_{nm} ( \la ))^\infty _{n,m = 1} $ as specified by \eqref{3.7} and \eqref{3.6}, respectively.  It is shown  on p.72 of \cite{CR3} that
 \begin{equation} \label{5.4} \textstyle
  | e_{nm} ( \la ) | \leq \frac{1}{n}, \quad 1 \leq m < n, \quad n \in \N .
  \end{equation}
  \textit{Warning}: \ In \cite{CR3} the set $ \N = \{ 0, 1, 2, \ldots \} $ is used rather than $ \N = \{ 1, 2, 3, \ldots \} $
  which is used here and so the
  inequalities from \cite{CR3} are   slightly different  when they are stated here. Back to our proof, it  is clear from \eqref{1.4} that the matrix
  $ A = ( c_{nm}) ^\infty_{ n, m = 1} $ for the Cesàro operator  $ C $ is lower triangular with its $n$-th row,
  for each $ n \in \N , $  given by $ c_{nm} := \frac{1}{n} $ for
  $ 1 \leq m \leq n $ and   $ c_{nm} := 0$ for $ m > n.$  Setting $ B := E_\la $ it is clear from \eqref{5.4} that \eqref{3.2}  is satisfied for the pair
  $ A, B $ in the space $ X := \cesp .$ Since $ C_{(p)}  = T_A : \cesp \lra \cesp $ is continuous, it follows from Corollary \ref{C3.2}
  that $ T_{E_\la} \in \cL^r ( \cesp) $ and hence, via Lemma \ref{L3.3} and \eqref{3.4}, that also $  ( C_{(p)} - \la I )^{-1} \in
  \cL^r ( \cesp) . $ \\

  Consider now the set $ \rho_2 .$ From \eqref{5.1} it is routine to establish that a non-zero point $ z \in \C$ belongs to $  \si ( C_{(p)} ) $ if and only if $ \mbox{Re} ( \frac{1}{z} ) \geq
  \frac{1}{p'}.$  From the case of equality in Remark \ref{R4.1},  it follows that
   $ \rho_2 = \bigcup_{ 0 < \alpha < 1 /p'}  \Ga _\alpha , $ where
   \begin{equation} \label{5.5} \textstyle
    \Ga_{\alpha}  := \left\{ z \in \C\!\ssm\!\{ 0 \} : \mbox{Re} \left( \frac{1}{z}\right) = \alpha \right\} =
    \left\{ z \in \C\!\ssm\!\{0\} : \left| z - \frac{1}{2 \alpha} \right| = \frac{1}{2 \alpha} \right\} .
    \end{equation}
    Fix a point  $ \la \in \rho_2 .$ Then there exists a unique number  $ \alpha \in ( 0, \frac{1}{p'} )$ such that
    $ \la \in \Ga_\alpha ,$ namely $ \alpha : = \mbox{Re}  ( \frac{1}{\la} ).$ In the notation of \eqref{3.7}
    it is shown on p.72 of \cite{CR3} that
    \begin{equation} \label{5.6} \textstyle
    | e_{nm}  ( \la ) | \leq e _{nm} \left( \frac{1}{\alpha}  \right) , \quad n, m \in \N .
    \end{equation}
    Note that $  e_{nm} ( \frac{1}{\alpha} ) \geq 0 $ for all $ n, m \in \N $ follows from  \eqref{3.7}  as $ 0 < \alpha < \frac{1}{p'} $
    implies that $ 1 - \frac{1}{k ( 1 / \alpha)}  = ( 1 - \frac{\alpha}{k}  ) > 0 $ for $ m \leq k \leq n .$
    Setting $ \widetilde{A} := E_{ 1 / \alpha} $ and $ \widetilde{B} := E_\lambda $ it is clear from \eqref{5.6} that \eqref{3.2} is
    satisfied for the pair $ \widetilde{A} , \widetilde{B}$ in place of $ A, B.$  Moreover, $ \frac{1}{\alpha} > p' $ implies that $ \frac{1}{\alpha}  \in \rho ( C_{(p)} ),$
    that is, $ ( C_{(p)} - \frac{1}{\alpha} I ) ^{-1} \in \cL ( \cesp ).$ Since $ T_{D _{1/\alpha}} \in \cL ( \cesp) $ by Lemma \ref{L3.3} (with
    $ \frac{1}{\alpha} $ in place of $ \lambda ),$ the identity $ T _{ E_{ 1 / \alpha}} = \alpha ^2 ( T_{D_{ 1 / \alpha}} -  ( C_{ (p)} - \frac{1}{\alpha} I ) ^{-1} ) $
    shows that  $ T_{\widetilde{A}} \in \cL ( \cesp).$ Hence, Corollary \ref{C3.2} can be applied to conclude that  $ T_{\widetilde{B}} = T _{ E_\la}  \in \cL ^r (\cesp) .$
    It then follows from \eqref{3.4} and Lemma 3.3 that $ ( C_{(p)} - \la I )^{-1} \in \cL ^r (\cesp). $
   \end{proof}
   The remaining space to  consider  is $ \ces (0)  := \{ x \in \C^{\N}  : C ( | x | ) \in c _0 \} $ equipped with the norm
   $$  \textstyle
   \| x \| _{\ces(0)}  := \| C (| x | ) \| _{c_0} = \sup _{n \in \N}  \frac{1}{n} \sum^n _{ k = 1} | x _k | , \quad
   x \in \ces(0).
   $$
   It is a Banach lattice  sequence space for the norm $ \| \cdot \| _{ \ces (0)} $ and the coordinatewise order. According to
   \cite[Theorem 6.4]{CR3}, the restriction of $ C$ (see (\eqref{1.4}) to  $ \ces(0),$ denoted here  by $ C_{(0)}, $
   is continuous with $ \| C _{(0)} \|_{\rm op} = 1 $ and
   \begin{equation}  \label{5.7} \textstyle
    \si (C _{(0)} ) = \{ \la \in \C : | \la - \frac{1}{2} | \leq \frac{1}{2} \} .
  \end{equation}
  \begin{Prop} \label{P5.2} \  The order spectrum of the positive operator  $ C_{(0)}  \in \cL ( \ces(0))$ satisfies
  $$
  \si _{\rm o} ( C_{(0)}) = \si ( C_{(0)}) .
  $$
  \end{Prop}
  \begin{proof} As usual it suffices to show that $ \rho ( C_{(0)} ) \subseteq \rho_{\rm o} ( C_{(0)} ).$ \\

   Let the set
  $ \rho _1 $ be as in \eqref{5.3}. For each $ \alpha \in (0,1)$ let $ \Ga_\alpha $ be given by  \eqref{5.5}. Then
  \eqref{5.7} ensures that we have the disjoint partition $ \rho ( C_{ (0)} ) = \rho_1 \cup \rho_2  $
  with $ \rho_2 := \bigcup _{ 0 < \alpha < 1 } \Ga_\alpha . $

  For any given point  $ \la \in \rho_1 $ the estimates \eqref{5.4} are again valid (see \cite[p.72]{CR3}) and so the argument in the proof
  of Proposition \ref{P5.1} can be easily adapted ( now for $ X := \ces (0) )$ to again show that  $ ( C_{(0)} - \lambda I)^{-1}
  \in \cL ^r ( \ces(0) ).$ \\

  Fix now $ \la \in \rho_2.$ Then there exists a unique $ \alpha \in (0,1)$ such that
  $ \la \in \Ga _\alpha ,$ namely $ \alpha := \mbox{Re} ( \frac{1}{\la}).$ Then $ \mbox{Re} (1 - \frac{1}{k \la} ) =  ( 1 - \frac{\alpha}{k} ) \geq 0 $
  for $ k \in \N .$ Arguing as at the bottom of p.396 in \cite{CR2}, now with $ x \in \ces(0)$ in place of  $ a \in \ces (2) $
  there, it follows that the $1\!$-st coordinate of $ E_\la ( x)$ is 0 and, for $ n \geq 2 ,$ that the $n$-th coordinate of $ E_\la (x)$ satisfies
  $$
  | ( E_\la ( x)) _n | \leq  ( E _{ 1 / \alpha} ( | x | ))_n , \quad x \in \ces (0).
  $$
  Substituting $ x := ( \delta _{rj} )^\infty _{ j = 1} $ into the previous estimates, for each $ r \in \N, $
  yields \eqref{5.6}. Since $ 0 < \alpha < 1 $ implies that $ \frac{1}{\alpha} \in \rho ( C_{(0)} ),$ the argument
  can be completed along the lines given in the proof of Proposition \ref{P5.1} to conclude that $ ( C_{(0)} - \la I )^{-1}
  \in \cL ^r ( \ces (0)).$ We  again warn the reader that $ \N = \{ 0, 1, 2, \ldots  \} $ is used in \cite{CR2}.
  \end{proof}

  \textbf{Acknowledgement.} \ The research of the first author (J. Bonet) was partially supported by the projects  MTM2016-76647-P and
  GV Prometeo 2017/102 (Spain).

  \bigskip

\bibliographystyle{plain}

  \end{document}